\documentclass[11pt]{amsart}
\usepackage{epsfig,amsfonts,amsthm,amssymb,latexsym,amsmath}

\numberwithin{equation}{section}

\newcommand{\mymod}[3]{#1 \equiv #2 \kern -0.5em \pmod{#3}}
\newcommand{\mynotmod}[3]{#1 \not \equiv #2 \kern -0.6em \pmod{#3}}

\theoremstyle{plain}
\newtheorem{theorem}{Theorem}[section]

\theoremstyle{remark}

\newtheorem{example}[theorem]{Example}
\theoremstyle{definition}

\title{Dual Third-order Jacobsthal Quaternions}
\vspace{5pt}

\author[G. Cerda-Morales]{Gamaliel Cerda-Morales}
\date{}

\begin{document}
\maketitle

\vspace{-20pt}
\begin{center}
{\footnotesize Instituto de Matem\'aticas, Pontificia Universidad Cat\'olica de Valpara\'iso, Blanco Viel 596, Valpara\'iso, Chile. \\
E-mails: gamaliel.cerda@usm.cl / gamaliel.cerd.m@mail.pucv.cl  
}\end{center}

\vspace{5pt}

\begin{abstract}
In 2016, Y\"uce and Torunbalc\i\ Ayd\i n \cite{Yuc-Tor} defined dual Fibonacci quaternions. In this paper, we defined the dual third-order Jacobsthal quaternions and dual third-order Jacobsthal-Lucas quaternions. Also, we investigated the relations between the dual third-order Jacobsthal quaternions and third-order Jacobsthal numbers. Furthermore, we gave some their quadratic properties, the summations, the Binet's formulas and Cassini-like identities for these quaternions.
\end{abstract}

\medskip
\noindent

\medskip
\noindent
\medskip

\section{Introduction}\label{sec:1}
The real quaternions are a number system which extends to the complex numbers. They are first described by Irish mathematician William Rowan Hamilton in 1843. In 1963, Horadam \cite{Hor2} defined the $n$-th Fibonacci quaternion which can be represented as
\begin{equation}\label{eq:1}
Q_{F}=\{Q_{n}=F_{n}+\textbf{i}F_{n+1}+\textbf{j}F_{n+2}+\textbf{k}F_{n+3}:\ \textrm{$F_{n}$ is $n$-th Fibonacci number} \},
\end{equation}
where $\textbf{i}^{2}=\textbf{j}^{2}=\textbf{k}^{2}=\textbf{ijk}=-1$.

In 1969, Iyer \cite{Iye1,Iye2} derived many relations for the Fibonacci quaternions. In 1977, Iakin \cite{Iak1,Iak2} introduced higher order quaternions and gave some identities for these quaternions. Furthermore, Horadam \cite{Hor3} extend to quaternions to the complex Fibonacci numbers defined by Harman \cite{Har}. In 2012, Hal\i c\i\ \cite{Hal} gave generating functions and Binet's formulas for Fibonacci and Lucas quaternions.

In 2006, Majernik \cite{Maj} defined a new type of quaternions, the so-called dual quaternions in the form $Q_{\mathbb{N}}=\{a+b\textbf{i}+c\textbf{j}+d\textbf{k}:\ a,b,c,d\in \mathbb{R}\}$, with the following multiplication schema for the quaternion units
\begin{equation}\label{eq:2}
\textbf{i}^{2}=\textbf{j}^{2}=\textbf{k}^{2}=0,\ \textbf{ij}=-\textbf{ji}=\textbf{jk}=-\textbf{kj}=\textbf{ik}=-\textbf{ki}=0.
\end{equation}
In 2009, Ata and Yayl\i\ \cite{Ata-Yay} defined dual quaternions with dual numbers coefficient as follows:
\begin{equation}
Q_{\mathbb{D}}=\{A+B\textbf{i}+C\textbf{j}+D\textbf{k}:\ A,B,C,D \in \mathbb{D},\ \textbf{i}^{2}=\textbf{j}^{2}=\textbf{k}^{2}=\textbf{ijk}=-1\},
\end{equation}
where $\mathbb{D}=\mathbb{R}[\varepsilon]=\{a+b\varepsilon:\ a,b\in \mathbb{R},\ \varepsilon^{2}=0,\ \varepsilon\neq 0 \}$. It is clear that $Q_{\mathbb{N}}$ and $Q_{\mathbb{D}}$ are different sets. In 2014, Nurkan and G\"uven \cite{Nur-Guv} defined dual Fibonacci quaternions as follows:
\begin{equation}\label{eq:3}
\mathbb{D}_{F}=\{Q_{n}=\widehat{F}_{n}+\textbf{i}\widehat{F}_{n+1}+\textbf{j}\widehat{F}_{n+2}+\textbf{k}\widehat{F}_{n+3}:\ \widehat{F}_{n}=F_{n}+\varepsilon F_{n+1}\},
\end{equation}
where $\textbf{i}^{2}=\textbf{j}^{2}=\textbf{k}^{2}=\textbf{ijk}=-1$ and $\widehat{F}_{n}$ is the $n$-th dual Fibonacci number. 

In 2016, Y\"uce and Torunbalc\i\ Ayd\i n \cite{Yuc-Tor} defined dual Fibonacci quaternions as follows:
\begin{equation}\label{eq:4}
\mathbb{N}_{F}=\{Q_{n}=F_{n}+\textbf{i}F_{n+1}+\textbf{j}F_{n+2}+\textbf{k}F_{n+3}:\ \textrm{$F_{n}$ is $n$-th Fibonacci number}\},
\end{equation}
where $\textbf{i}^{2}=\textbf{j}^{2}=\textbf{k}^{2}=0,\ \textbf{ij}=-\textbf{ji}=\textbf{jk}=-\textbf{kj}=\textbf{ik}=-\textbf{ki}=0$.

In the other hand, the Jacobsthal numbers have many interesting properties and applications in many fields of science (see, e.g., \cite{Bar}). The Jacobsthal numbers $J_{n}$ are defined by the recurrence relation
\begin{equation}\label{eq:5}
J_{0}=0,\ J_{1}=1,\ J_{n+1}=J_{n}+2J_{n-1},\ n\geq1.
\end{equation}
Another important sequence is the Jacobsthal-Lucas sequence. This sequence is defined by the recurrence relation $j_{n+1}=j_{n}+2j_{n-1},\ n\geq1$ and $j_{0}=2,\ j_{1}=1 $. (see, \cite{Hor4}).

In \cite{Coo-Bac} the Jacobsthal recurrence relation (\ref{eq:5}) is extended to higher order recurrence relations and the basic list of identities provided by A. F. Horadam \cite{Hor4} is expanded and extended to several identities for some of the higher order cases. In particular, third-order Jacobsthal numbers, $\{J_{n}^{(3)}\}_{n\geq0}$, and third-order Jacobsthal-Lucas numbers, $\{j_{n}^{(3)}\}_{n\geq0}$, are defined by
\begin{equation}\label{eq:6}
J_{n+3}^{(3)}=J_{n+2}^{(3)}+J_{n+1}^{(3)}+2J_{n}^{(3)},\ J_{0}^{(3)}=0,\ J_{1}^{(3)}=J_{2}^{(3)}=1,\ n\geq0,
\end{equation}
and 
\begin{equation}\label{eq:7}
j_{n+3}^{(3)}=j_{n+2}^{(3)}+j_{n+1}^{(3)}+2j_{n}^{(3)},\ j_{0}^{(3)}=2,\ j_{1}^{(3)}=1,\ j_{2}^{(3)}=5,\ n\geq0,
\end{equation}
respectively.

The following properties given for third-order Jacobsthal numbers and third-order Jacobsthal-Lucas numbers play important roles in this paper (for more, see \cite{Coo-Bac}). 
\begin{equation}\label{p:1}
3J_{n}^{(3)}+j_{n}^{(3)}=2^{n+1},
\end{equation}
\begin{equation}\label{p:2}
j_{n}^{(3)}-3J_{n}^{(3)}=2j_{n-3}^{(3)},
\end{equation}
\begin{equation}\label{p:3}
J_{n+2}^{(3)}-4J_{n}^{(3)}=\left\{ 
\begin{array}{ccc}
-2 & \textrm{if} & \mymod{n}{1}{3} \\ 
1 & \textrm{if} & \mynotmod{n}{1}{3}%
\end{array}%
\right. ,
\end{equation}
\begin{equation}\label{p:4}
j_{n}^{(3)}-4J_{n}^{(3)}=\left\{ 
\begin{array}{ccc}
2 & \textrm{if} & \mymod{n}{0}{3} \\ 
-3 & \textrm{if} & \mymod{n}{1}{3}\\ 
1 & \textrm{if} & \mymod{n}{2}{3}%
\end{array}%
\right. ,
\end{equation}
\begin{equation}\label{p:5}
j_{n+1}^{(3)}+j_{n}^{(3)}=3J_{n+2}^{(3)},
\end{equation}
\begin{equation}\label{p:6}
j_{n}^{(3)}-J_{n+2}^{(3)}=\left\{ 
\begin{array}{ccc}
1 & \textrm{if} & \mymod{n}{0}{3} \\ 
-1 & \textrm{if} & \mymod{n}{1}{3} \\ 
0 & \textrm{if} & \mymod{n}{2}{3}%
\end{array}%
\right. ,
\end{equation}
\begin{equation}\label{p:7}
\left( j_{n-3}^{(3)}\right) ^{2}+3J_{n}^{(3)}j_{n}^{(3)}=4^{n},
\end{equation}
\begin{equation}\label{p:8}
\sum\limits_{k=0}^{n}J_{k}^{(3)}=\left\{ 
\begin{array}{ccc}
J_{n+1}^{(3)} & \textrm{if} & \mynotmod{n}{0}{3} \\ 
J_{n+1}^{(3)}-1 & \textrm{if} & \mymod{n}{0}{3}%
\end{array}%
\right. 
\end{equation}
and
\begin{equation}\label{p:9}
\left( j_{n}^{(3)}\right) ^{2}-9\left( J_{n}^{(3)}\right)^{2}=2^{n+2}j_{n-3}^{(3)}.
\end{equation}

Using standard techniques for solving recurrence relations, the auxiliary equation, and its roots are given by 
$$x^{3}-x^{2}-x-2=0;\ x = 2,\ \textrm{and}\ x=\frac{-1\pm i\sqrt{3}}{2}.$$ 

Note that the latter two are the complex conjugate cube roots of unity. Call them $\omega_{1}$ and $\omega_{2}$, respectively. Thus the Binet formulas can be written as
\begin{equation}\label{b1}
J_{n}^{(3)}=\frac{2}{7}2^{n}-\frac{3+2i\sqrt{3}}{21}\omega_{1}^{n}-\frac{3-2i\sqrt{3}}{21}\omega_{2}^{n}=\frac{1}{7}\left(2^{n+1}-V_{n}^{(3)}\right)
\end{equation}
and
\begin{equation}\label{b2}
j_{n}^{(3)}=\frac{8}{7}2^{n}+\frac{3+2i\sqrt{3}}{7}\omega_{1}^{n}+\frac{3-2i\sqrt{3}}{7}\omega_{2}^{n}=\frac{1}{7}\left(2^{n+3}+3V_{n}^{(3)}\right),
\end{equation}
respectively. Here $V_{n}^{(3)}$ is the sequence defined by 
\begin{equation}\label{v}
V_{n}^{(3)}=\frac{3+2i\sqrt{3}}{3}\omega_{1}^{n}+\frac{3-2i\sqrt{3}}{3}\omega_{2}^{n}=\left\{ 
\begin{array}{ccc}
2 & \textrm{if} & \mymod{n}{0}{3} \\ 
-3 & \textrm{if} & \mymod{n}{1}{3} \\ 
1 & \textrm{if} & \mymod{n}{2}{3}%
\end{array}%
\right. .
\end{equation}

Recently in \cite{Cer}, we have defined a new type of quaternions with the third-order Jacobsthal and third-order Jacobsthal-Lucas number components as
\[
JQ_{n}^{(3)}=J_{n}^{(3)}+J_{n+1}^{(3)}\textbf{i}+J_{n+2}^{(3)}\textbf{j}+J_{n+3}^{(3)}\textbf{k}
\]
and
\[
jQ_{n}^{(3)}=j_{n}^{(3)}+j_{n+1}^{(3)}\textbf{i}+j_{n+2}^{(3)}\textbf{j}+j_{n+3}^{(3)}\textbf{k},
\]
respectively, where $\textbf{i}^{2}=\textbf{j}^{2}=\textbf{k}^{2}=\textbf{ijk}=-1$, and we studied the properties of these quaternions. Also, we derived the generating functions and many other identities for the third-order Jacobsthal and third-order Jacobsthal-Lucas quaternions.

In this paper, we define the dual third-order Jacobsthal quaternions and dual third-order Jacobsthal-Lucas quaternions as follows:
\begin{equation}\label{def:1}
JN_{m}^{(3)}=J_{m}^{(3)}+J_{m+1}^{(3)}\textbf{i}+J_{m+2}^{(3)}\textbf{j}+J_{m+3}^{(3)}\textbf{k}\ (m\geq0)
\end{equation}
and
\begin{equation}\label{def:2}
jN_{m}^{(3)}=j_{m}^{(3)}+j_{m+1}^{(3)}\textbf{i}+j_{m+2}^{(3)}\textbf{j}+j_{m+3}^{(3)}\textbf{k}\ (m\geq0),
\end{equation}
respectively. Here $\textbf{i}^{2}=\textbf{j}^{2}=\textbf{k}^{2}=0,\ \textbf{ij}=-\textbf{ji}=\textbf{jk}=-\textbf{kj}=\textbf{ik}=-\textbf{ki}=0$. Also, we investigated the relations between the dual third-order Jacobsthal quaternions and third-order Jacobsthal numbers. Furthermore, we gave some their quadratic properties, the Binet's formulas, d'Ocagne and Cassini-like identities for these quaternions.

\section{Dual Third-order Jacobsthal Quaternions}\label{sec:2}
We can define dual third-order Jacobsthal quaternions by using third-order Jacobsthal numbers. The $n$-th third-order Jacobsthal number $J_{n}^{(3)}$ is defined by Eq. (\ref{eq:6}). Then, we can define the dual third-order Jacobsthal quaternions as follows:
\begin{equation}\label{eq:8}
\mathbb{N}_{J}=\{JN_{m}^{(3)}=J_{m}^{(3)}+J_{m+1}^{(3)}\textbf{i}+J_{m+2}^{(3)}\textbf{j}+J_{m+3}^{(3)}\textbf{k}\},
\end{equation}
where $J_{m}^{(3)}$ is the $m$-th third-order Jacobsthal number and $\{\textbf{i},\textbf{j},\textbf{k}\}$ as in Eq. (\ref{eq:2}). Also, we can define the dual third-order Jacobsthal-Lucas quaternion as follows:
\begin{equation}\label{eq:9}
\mathbb{N}_{j}=\{jN_{m}^{(3)}=j_{m}^{(3)}+j_{m+1}^{(3)}\textbf{i}+j_{m+2}^{(3)}\textbf{j}+j_{m+3}^{(3)}\textbf{k}\},
\end{equation}
where $j_{m}^{(3)}$ is the $m$-th third-order Jacobsthal-Lucas number.

Then, the addition and subtraction of the dual third-order Jacobsthal and dual third-order Jacobsthal-Lucas quaternions is defined by
\begin{equation}\label{eq:10}
\begin{aligned}
JN_{m}^{(3)}\pm jN_{m}^{(3)}&=(J_{m}^{(3)}+J_{m+1}^{(3)}\textbf{i}+J_{m+2}^{(3)}\textbf{j}+J_{m+3}^{(3)}\textbf{k})\\
&\ \ \pm (j_{m}^{(3)}+j_{m+1}^{(3)}\textbf{i}+j_{m+2}^{(3)}\textbf{j}+j_{m+3}^{(3)}\textbf{k})\\
&=(J_{m}^{(3)}\pm j_{m}^{(3)})+(J_{m+1}^{(3)}\pm j_{m+1}^{(3)})\textbf{i}+(J_{m+2}^{(3)}\pm j_{m+2}^{(3)})\textbf{j}\\
&\ \ +(J_{m+3}^{(3)}\pm j_{m+3}^{(3)})\textbf{k}
\end{aligned}
\end{equation}
and the multiplication of the dual third-order Jacobsthal and dual third-order Jacobsthal-Lucas quaternions is defined by
\begin{equation}\label{eq:11}
\begin{aligned}
&JN_{m}^{(3)}jN_{m}^{(3)}\\
&=(J_{m}^{(3)}+J_{m+1}^{(3)}\textbf{i}+J_{m+2}^{(3)}\textbf{j}+J_{m+3}^{(3)}\textbf{k})(j_{m}^{(3)}+j_{m+1}^{(3)}\textbf{i}+j_{m+2}^{(3)}\textbf{j}+j_{m+3}^{(3)}\textbf{k})\\
&=J_{m}^{(3)}j_{m}^{(3)}+(J_{m}^{(3)}j_{m+1}^{(3)}+J_{m+1}^{(3)}j_{m}^{(3)})\textbf{i}+(J_{m}^{(3)}j_{m+2}^{(3)}+J_{m+2}^{(3)}j_{m}^{(3)})\textbf{j}\\
&\ \ +(J_{m}^{(3)}j_{m+3}^{(3)}+J_{m+3}^{(3)}j_{m}^{(3)})\textbf{k}.
\end{aligned}
\end{equation}

Now, the scalar and the vector part of the $JN_{m}^{(3)}$ which is the $m$-th term of the dual third-order Jacobsthal sequence $\{JN_{m}^{(3)}\}_{m\geq0}$ are denoted by
\begin{equation}\label{eq:12}
(S_{JN_{m}^{(3)}},V_{JN_{m}^{(3)}})=(J_{m}^{(3)},J_{m+1}^{(3)}\textbf{i}+J_{m+2}^{(3)}\textbf{j}+J_{m+3}^{(3)}\textbf{k}).
\end{equation}
Thus, the dual third-order Jacobsthal $JN_{m}^{(3)}$ is given by $S_{JN_{m}^{(3)}}+V_{JN_{m}^{(3)}}$. Then, relation (\ref{eq:11}) is defined by
\begin{equation}\label{eq:13}
JN_{m}^{(3)}jN_{m}^{(3)}=S_{JN_{m}^{(3)}}S_{jN_{m}^{(3)}}+S_{JN_{m}^{(3)}}V_{jN_{m}^{(3)}}+S_{jN_{m}^{(3)}}V_{JN_{m}^{(3)}}.
\end{equation}
The conjugate of dual third-order Jacobsthal quaternion $JN_{m}^{(3)}$ is denoted by $\overline{JN}_{m}^{(3)}$ and
it is $\overline{JN}_{m}^{(3)}=J_{m}^{(3)}-J_{m+1}^{(3)}\textbf{i}-J_{m+2}^{(3)}\textbf{j}-J_{m+3}^{(3)}\textbf{k}$. The norm of $JN_{m}^{(3)}$ is defined as
\begin{equation}\label{eq:15}
\Vert JN_{m}^{(3)} \Vert ^{2}=JN_{m}^{(3)}\overline{JN}_{m}^{(3)}=\overline{JN}_{m}^{(3)}JN_{m}^{(3)}=\left(J_{m}^{(3)}\right)^{2}.
\end{equation}

Then, we give the following theorem using statements (\ref{eq:8}) and (\ref{eq:9}).
\begin{theorem}\label{th:1}
Let $J_{m}^{(3)}$ and $JN_{m}^{(3)}$ be the $m$-th terms of the third-order Jacobsthal sequence $\{J_{m}^{(3)}\}_{m\geq0}$ and the dual third-order Jacobsthal quaternion sequence $\{JN_{m}^{(3)}\}_{m\geq0}$, respectively. In this case, for $m\geq 0$ we can give the following relations:
\begin{equation}\label{t1}
2JN_{m}^{(3)}+JN_{m+1}^{(3)}+JN_{m+2}^{(3)}=JN_{m+3}^{(3)},
\end{equation}
\begin{equation}\label{t2}
JN_{m}^{(3)}-JN_{m+1}^{(3)}\textbf{i}-JN_{m+2}^{(3)}\textbf{j}-JN_{m+3}^{(3)}\textbf{k}=J_{m}^{(3)},
\end{equation}
\begin{equation}\label{t3}
\Vert JN_{m}^{(3)} \Vert ^{2}+\Vert JN_{m+1}^{(3)} \Vert ^{2}+\Vert JN_{m+2}^{(3)} \Vert ^{2}=\frac{1}{7}\left(\begin{array}{ccc} 3\cdot2^{2(m+1)}(1+4\textbf{i}+8\textbf{j}+16\textbf{k})\\
-2^{m+2}UN_{m}^{(3)}\\
-2^{m+3}U_{m}^{(3)}(\textbf{i}+2\textbf{j}+4\textbf{k})\\
+2(1-\textbf{i}-\textbf{j}+2\textbf{k})\end{array}\right),
\end{equation}
where $UN_{m}^{(3)}=U_{m}^{(3)}+U_{m+1}^{(3)}\textbf{i}+U_{m+2}^{(3)}\textbf{j}+U_{m+3}^{(3)}\textbf{k}$ and $U_{m}^{(3)}=\frac{1}{7}\left(V_{m+1}^{(3)}+3V_{m+2}^{(3)}\right)$.
\end{theorem}
\begin{proof}
(\ref{t1}): By the equations $JN_{m}^{(3)}=J_{m}^{(3)}+J_{m+1}^{(3)}\textbf{i}+J_{m+2}^{(3)}\textbf{j}+J_{m+3}^{(3)}\textbf{k}$ and (\ref{eq:6}), we get
\begin{align*}
2JN_{m}^{(3)}&+JN_{m+1}^{(3)}+JN_{m+2}^{(3)}\\
&=(2J_{m}^{(3)}+2J_{m+1}^{(3)}\textbf{i}+2J_{m+2}^{(3)}\textbf{j}+2J_{m+3}^{(3)}\textbf{k})\\
&\ \ + (J_{m+1}^{(3)}+J_{m+2}^{(3)}\textbf{i}+J_{m+3}^{(3)}\textbf{j}+J_{m+4}^{(3)}\textbf{k})\\
&\ \ + (J_{m+2}^{(3)}+J_{m+3}^{(3)}\textbf{i}+J_{m+4}^{(3)}\textbf{j}+J_{m+5}^{(3)}\textbf{k})\\
&=(2J_{m}^{(3)}+J_{m+1}^{(3)}+J_{m+2}^{(3)})+(2J_{m+1}^{(3)}+J_{m+2}^{(3)}+J_{m+3}^{(3)})\textbf{i}\\
&\ \ + (2J_{m+2}^{(3)}+J_{m+3}^{(3)}+J_{m+4}^{(3)})\textbf{j}+(2J_{m+3}^{(3)}+J_{m+4}^{(3)}+J_{m+5}^{(3)})\textbf{k})\\
&=J_{m+3}^{(3)}+J_{m+4}^{(3)}\textbf{i}+J_{m+5}^{(3)}\textbf{j}+J_{m+6}^{(3)}\textbf{k}=JN_{m+3}^{(3)}.
\end{align*}
(\ref{t2}): By using $JN_{m}^{(3)}$ in the Eq. (\ref{eq:6}) and conditions (\ref{eq:2}), we get
\begin{align*}
JN_{m}^{(3)}-JN_{m+1}^{(3)}\textbf{i}-JN_{m+2}^{(3)}\textbf{j}-JN_{m+3}^{(3)}\textbf{k}&=J_{m}^{(3)}+J_{m+1}^{(3)}\textbf{i}+J_{m+2}^{(3)}\textbf{j}+J_{m+3}^{(3)}\textbf{k}\\
&\ \ - (J_{m+1}^{(3)}+J_{m+2}^{(3)}\textbf{i}+J_{m+3}^{(3)}\textbf{j}+J_{m+4}^{(3)}\textbf{k})\textbf{i}\\
&\ \ - (J_{m+2}^{(3)}+J_{m+3}^{(3)}\textbf{i}+J_{m+4}^{(3)}\textbf{j}+J_{m+5}^{(3)}\textbf{k})\textbf{j}\\
&\ \ - (J_{m+3}^{(3)}+J_{m+4}^{(3)}\textbf{i}+J_{m+5}^{(3)}\textbf{j}+J_{m+6}^{(3)}\textbf{k})\textbf{k}\\
&=J_{m}^{(3)}.
\end{align*}
(\ref{t3}): By using Eqs. (\ref{eq:11}) and (\ref{b1}), we get
\begin{equation}\label{eq:16}
\left(JN_{m}^{(3)}\right)^{2}=\left(J_{m}^{(3)}\right)^{2}+2J_{m}^{(3)}J_{m+1}^{(3)}\textbf{i}+2J_{m}^{(3)}J_{m+2}^{(3)}\textbf{j}+2J_{m}^{(3)}J_{m+3}^{(3)}\textbf{k},
\end{equation}
and
\begin{equation}\label{eq:17}
\begin{aligned}
&\left(J_{m}^{(3)}\right)^{2}+\left(J_{m+1}^{(3)}\right)^{2}+\left(J_{m+2}^{(3)}\right)^{2}\\
&=\frac{1}{49}\left( \left(2^{m+1}-V_{m}^{(3)}\right)^{2}+\left(2^{m+2}-V_{m+1}^{(3)}\right)^{2}+\left(2^{m+3}-V_{m+2}^{(3)}\right)^{2}\right)\\
&=\frac{1}{49}\left( 21\cdot 2^{2(m+1)}-2^{m+2}(V_{m}^{(3)}+2V_{m+1}^{(3)}+4V_{m+2}^{(3)})+14\right)\\
&=\frac{1}{7}\left( 3\cdot 2^{2(m+1)}-2^{m+2}U_{m}^{(3)}+2\right),
\end{aligned}
\end{equation}
where $U_{m}^{(3)}=\frac{1}{7}\left(V_{m+1}^{(3)}+3V_{m+2}^{(3)}\right)$. Finally, from the Eqs. (\ref{eq:16}) and (\ref{eq:17}), we obtain
\begin{align*}
\left(JN_{m}^{(3)}\right)^{2}+&\left(JN_{m+1}^{(3)}\right)^{2}+\left(JN_{m+2}^{(3)}\right)^{2}\\
&=\left(J_{m}^{(3)}\right)^{2}+\left(J_{m+1}^{(3)}\right)^{2}+\left(J_{m+2}^{(3)}\right)^{2}\\
&\ \ +2\left(J_{m}^{(3)}J_{m+1}^{(3)}+J_{m+1}^{(3)}J_{m+2}^{(3)}+J_{m+2}^{(3)}J_{m+3}^{(3)}\right)\textbf{i}\\
&\ \ +2\left(J_{m}^{(3)}J_{m+2}^{(3)}+J_{m+1}^{(3)}J_{m+3}^{(3)}+J_{m+2}^{(3)}J_{m+4}^{(3)}\right)\textbf{j}\\
&\ \ +2\left(J_{m}^{(3)}J_{m+3}^{(3)}+J_{m+1}^{(3)}J_{m+4}^{(3)}+J_{m+2}^{(3)}J_{m+5}^{(3)}\right)\textbf{k}\\
&=\frac{1}{7}\left(\begin{array}{ccc} 3\cdot2^{2(m+1)}(1+4\textbf{i}+8\textbf{j}+16\textbf{k})-2^{m+2}UN_{m}^{(3)}\\
-2^{m+3}U_{m}^{(3)}(\textbf{i}+2\textbf{j}+4\textbf{k})+2(1-\textbf{i}-\textbf{j}+2\textbf{k})\end{array}\right),
\end{align*}
where $UN_{m}^{(3)}=U_{m}^{(3)}+U_{m+1}^{(3)}\textbf{i}+U_{m+2}^{(3)}\textbf{j}+U_{m+3}^{(3)}\textbf{k}$.
\end{proof}

\begin{theorem}\label{th:2}
Let $JN_{m}^{(3)}$ and $jN_{m}^{(3)}$ be the $m$-th terms of the dual third-order Jacobsthal quaternion sequence $\{JN_{m}^{(3)}\}_{m\geq0}$ and the dual third-order Jacobsthal-Lucas quaternion sequence $\{jN_{m}^{(3)}\}_{m\geq0}$, respectively. The following relations are satisfied
\begin{equation}\label{t4}
jN_{m+3}^{(3)}-3JN_{m+3}^{(3)}=2jN_{m}^{(3)},
\end{equation}
\begin{equation}\label{t5}
jN_{m+1}^{(3)}+jN_{m}^{(3)}=3JN_{m+2}^{(3)},
\end{equation}
\begin{equation}\label{t6}
\left(jN_{m}^{(3)}\right)^{2}+3JN_{m+3}^{(3)}jN_{m+3}^{(3)}=4^{m+3}(1+4\textbf{i}+8\textbf{j}+16\textbf{k}).
\end{equation}
\end{theorem}
\begin{proof}
(\ref{t4}): From identities between third-order Jacobsthal number and third-order Jacobsthal-Lucas number (\ref{p:2}) and (\ref{eq:10}), it follows that
\begin{align*}
jN_{m+3}^{(3)}-3JN_{m+3}^{(3)}&=j_{m+3}^{(3)}+j_{m+4}^{(3)}\textbf{i}+j_{m+5}^{(3)}\textbf{j}+j_{m+6}^{(3)}\textbf{k}\\
&\ \ -3(J_{m+3}^{(3)}+J_{m+4}^{(3)}\textbf{i}+J_{m+5}^{(3)}\textbf{j}+J_{m+6}^{(3)}\textbf{k})\\
&=(j_{m+3}^{(3)}-3J_{m+3}^{(3)})+(j_{m+4}^{(3)}-3J_{m+4}^{(3)})\textbf{i}\\
&\ \ + (j_{m+5}^{(3)}-3J_{m+5}^{(3)})\textbf{j}+(j_{m+6}^{(3)}-3J_{m+6}^{(3)})\textbf{k}\\
&=2j_{m}^{(3)}+2j_{m+1}^{(3)}\textbf{i}+2j_{m+2}^{(3)}\textbf{j}+2j_{m+3}^{(3)}\textbf{k}\\
&=2jN_{m}^{(3)}.
\end{align*}
The proof of (\ref{t5}) is similar to (\ref{t4}), using the identity (\ref{p:5}). 

(\ref{t6}): Now, using Eqs. (\ref{eq:11}), (\ref{eq:16}) and (\ref{p:7}), we get
\begin{align*}
\left(jN_{m}^{(3)}\right)^{2}+&3JN_{m+3}^{(3)}jN_{m+3}^{(3)}\\
&=\left(j_{m}^{(3)}\right)^{2}+2j_{m}^{(3)}j_{m+1}^{(3)}\textbf{i}+2j_{m}^{(3)}j_{m+2}^{(3)}\textbf{j}+2j_{m}^{(3)}j_{m+3}^{(3)}\textbf{k}\\
&\ \ + 3J_{m+3}^{(3)}j_{m+3}^{(3)}+3(J_{m+3}^{(3)}j_{m+4}^{(3)}+J_{m+4}^{(3)}j_{m+3}^{(3)})\textbf{i}\\
&\ \ +3(J_{m+3}^{(3)}j_{m+5}^{(3)}+J_{m+5}^{(3)}j_{m+3}^{(3)})\textbf{j}+3(J_{m+3}^{(3)}j_{m+6}^{(3)}+J_{m+6}^{(3)}j_{m+3}^{(3)})\textbf{k}\\
&=\left(j_{m}^{(3)}\right)^{2}+3J_{m+3}^{(3)}j_{m+3}^{(3)}\\
&\ \ +(2j_{m}^{(3)}j_{m+1}^{(3)}+3(J_{m+3}^{(3)}j_{m+4}^{(3)}+J_{m+4}^{(3)}j_{m+3}^{(3)}))\textbf{i}\\
&\ \ +(2j_{m}^{(3)}j_{m+2}^{(3)}+3(J_{m+3}^{(3)}j_{m+5}^{(3)}+J_{m+5}^{(3)}j_{m+3}^{(3)}))\textbf{j}\\
&\ \ +(2j_{m}^{(3)}j_{m+3}^{(3)}+3(J_{m+3}^{(3)}j_{m+6}^{(3)}+J_{m+6}^{(3)}j_{m+3}^{(3)}))\textbf{k}\\
&=4^{m+3}(1+4\textbf{i}+8\textbf{j}+16\textbf{k}).
\end{align*}
\end{proof}
\begin{theorem}\label{th:3}
Let $JN_{m}^{(3)}$ be the $m$-th term of the dual third-order Jacobsthal quaternion sequence $\{JN_{m}^{(3)}\}_{m\geq0}$. Then, we have the following identity
\begin{equation}\label{t7}
\sum_{s=0}^{m}JN_{s}^{(3)}=JN_{m+1}^{(3)}-\frac{1}{21}\left(7(1+\textbf{i}+4\textbf{j}+7\textbf{k})-4VN_{m+1}^{(3)}+VN_{m}^{(3)}\right),
\end{equation}
where $VN_{m}^{(3)}=V_{m}^{(3)}+V_{m+1}^{(3)}\textbf{i}+V_{m+2}^{(3)}\textbf{j}+V_{m+3}^{(3)}\textbf{k}$.
\end{theorem}
\begin{proof}
Since 
\begin{equation}\label{eq:18}
\begin{aligned}
\sum_{s=0}^{m}J_{s}^{(3)}&=J_{m+1}^{(3)}-\frac{1}{21}\left(7-4V_{m+1}^{(3)}+V_{m}^{(3)}\right)\\
&=\left\{ 
\begin{array}{ccc}
J_{m+1}^{(3)} & \textrm{if} & \mynotmod{m}{0}{3} \\ 
J_{m+1}^{(3)}-1 & \textrm{if} & \mymod{m}{0}{3}%
\end{array}%
\right.
\end{aligned}
\end{equation}
(see \cite{Coo-Bac}), we get
\begin{align*}
\sum_{s=0}^{m}JN_{s}^{(3)}&=\sum_{s=0}^{m}J_{s}^{(3)}+\textbf{i}\sum_{s=1}^{m+1}J_{s}^{(3)}+\textbf{j}\sum_{s=2}^{m+2}J_{s}^{(3)}+\textbf{k}\sum_{s=3}^{m+3}J_{s}^{(3)}\\
&=J_{m+1}^{(3)}-\frac{1}{21}\left(7-4V_{m+1}^{(3)}+V_{m}^{(3)}\right)\\
&\ \ +\left(J_{m+2}^{(3)}-\frac{1}{21}\left(7-4V_{m+2}^{(3)}+V_{m+1}^{(3)}\right)\right)\textbf{i}\\
&\ \ +\left(J_{m+3}^{(3)}-\frac{1}{21}\left(28-4V_{m+3}^{(3)}+V_{m+2}^{(3)}\right)\right)\textbf{j}\\
&\ \ +\left(J_{m+4}^{(3)}-\frac{1}{21}\left(49-4V_{m+4}^{(3)}+V_{m+3}^{(3)}\right)\right)\textbf{k}\\
&=JN_{m+1}^{(3)}-\frac{1}{21}\left(7(1+\textbf{i}+4\textbf{j}+7\textbf{k})-4VN_{m+1}^{(3)}+VN_{m}^{(3)}\right),
\end{align*}
where $VN_{m}^{(3)}=V_{m}^{(3)}+V_{m+1}^{(3)}\textbf{i}+V_{m+2}^{(3)}\textbf{j}+V_{m+3}^{(3)}\textbf{k}$.
\end{proof}

\begin{theorem}\label{th:4}
Let $JN_{m}^{(3)}$ and $jN_{m}^{(3)}$ be the $m$-th terms of the dual third-order Jacobsthal quaternion sequence $\{JN_{m}^{(3)}\}_{m\geq0}$ and the dual third-order Jacobsthal-Lucas quaternion sequence $\{jN_{m}^{(3)}\}_{m\geq0}$, respectively. Then, we have
\begin{equation}\label{t8}
jN_{m}^{(3)}\overline{JN}_{m}^{(3)}-\overline{jN}_{m}^{(3)}JN_{m}^{(3)}=2(J_{m}^{(3)}jN_{m}^{(3)}-j_{m}^{(3)}JN_{m}^{(3)}),
\end{equation}
\begin{equation}\label{t9}
jN_{m}^{(3)}JN_{m}^{(3)}+\overline{jN}_{m}^{(3)}\overline{JN}_{m}^{(3)}=2j_{m}^{(3)}JN_{m}^{(3)}.
\end{equation}
\end{theorem}
\begin{proof}
(\ref{t8}): By the Eqs. (\ref{eq:11}) and (\ref{eq:13}), we get
\begin{align*}
&jN_{m}^{(3)}\overline{JN}_{m}^{(3)}-\overline{jN}_{m}^{(3)}JN_{m}^{(3)}\\
&=(j_{m}^{(3)}+j_{m+1}^{(3)}\textbf{i}+j_{m+2}^{(3)}\textbf{j}+j_{m+3}^{(3)}\textbf{k})(J_{m}^{(3)}-J_{m+1}^{(3)}\textbf{i}-J_{m+2}^{(3)}\textbf{j}-J_{m+3}^{(3)}\textbf{k})\\
&\ \ - (j_{m}^{(3)}-j_{m+1}^{(3)}\textbf{i}-j_{m+2}^{(3)}\textbf{j}-j_{m+3}^{(3)}\textbf{k})(J_{m}^{(3)}+J_{m+1}^{(3)}\textbf{i}+J_{m+2}^{(3)}\textbf{j}+J_{m+3}^{(3)}\textbf{k})\\
&=2J_{m}^{(3)}(j_{m+1}^{(3)}\textbf{i}+j_{m+2}^{(3)}\textbf{j}+j_{m+3}^{(3)}\textbf{k})-2j_{m}^{(3)}(J_{m+1}^{(3)}\textbf{i}+J_{m+2}^{(3)}\textbf{j}+J_{m+3}^{(3)}\textbf{k})\\
&=2(J_{m}^{(3)}jN_{m}^{(3)}-j_{m}^{(3)}JN_{m}^{(3)}).
\end{align*}
(\ref{t9}):
\begin{align*}
&jN_{m}^{(3)}JN_{m}^{(3)}+\overline{jN}_{m}^{(3)}\overline{JN}_{m}^{(3)}\\
&=(j_{m}^{(3)}+j_{m+1}^{(3)}\textbf{i}+j_{m+2}^{(3)}\textbf{j}+j_{m+3}^{(3)}\textbf{k})(J_{m}^{(3)}+J_{m+1}^{(3)}\textbf{i}+J_{m+2}^{(3)}\textbf{j}+J_{m+3}^{(3)}\textbf{k})\\
&\ \ + (j_{m}^{(3)}-j_{m+1}^{(3)}\textbf{i}-j_{m+2}^{(3)}\textbf{j}-j_{m+3}^{(3)}\textbf{k})(J_{m}^{(3)}-J_{m+1}^{(3)}\textbf{i}-J_{m+2}^{(3)}\textbf{j}-J_{m+3}^{(3)}\textbf{k})\\
&=j_{m}^{(3)}J_{m}^{(3)}+(j_{m}^{(3)}J_{m+1}^{(3)}+j_{m+1}^{(3)}J_{m}^{(3)})\textbf{i}+(j_{m}^{(3)}J_{m+2}^{(3)}+j_{m+2}^{(3)}J_{m}^{(3)})\textbf{j}\\
&\ \ +(j_{m}^{(3)}J_{m+3}^{(3)}+j_{m+3}^{(3)}J_{m}^{(3)})\textbf{k}\\
&\ \ +j_{m}^{(3)}J_{m}^{(3)}-(j_{m}^{(3)}J_{m+1}^{(3)}+j_{m+1}^{(3)}J_{m}^{(3)})\textbf{i}-(j_{m}^{(3)}J_{m+2}^{(3)}+j_{m+2}^{(3)}J_{m}^{(3)})\textbf{j}\\
&\ \ -(j_{m}^{(3)}J_{m+3}^{(3)}+j_{m+3}^{(3)}J_{m}^{(3)})\textbf{k}\\
&=2j_{m}^{(3)}JN_{m}^{(3)}.
\end{align*} 
\end{proof}

\begin{theorem}[Binet's Formulas]\label{th:5}
Let $JN_{m}^{(3)}$ and $jN_{m}^{(3)}$ be $m$-th terms of the dual third-order Jacobsthal quaternion sequence $\{JN_{m}^{(3)}\}_{m\geq0}$ and the dual third-order Jacobsthal-Lucas quaternion sequence $\{jN_{m}^{(3)}\}_{m\geq0}$, respectively. For $m\geq0$, the Binet's formulas for these quaternions are as follows:
\begin{equation}\label{t10}
\begin{aligned}
JN_{m}^{(3)}&=\frac{2}{7}2^{m}\underline{\alpha}-\frac{3+2i\sqrt{3}}{21}\underline{\omega_{1}}\omega_{1}^{m}-\frac{3-2i\sqrt{3}}{21}\underline{\omega_{2}}\omega_{2}^{m}\\
&=\frac{1}{7}\left(2^{m+1}\underline{\alpha}-VN_{m}^{(3)}\right)
\end{aligned}
\end{equation}
and
\begin{equation}\label{t11}
\begin{aligned}
jN_{m}^{(3)}&=\frac{8}{7}2^{m}\underline{\alpha}+\frac{3+2i\sqrt{3}}{7}\underline{\omega_{1}}\omega_{1}^{m}+\frac{3-2i\sqrt{3}}{7}\underline{\omega_{2}}\omega_{2}^{m}\\
&=\frac{1}{7}\left(2^{m+3}\underline{\alpha}+3VN_{m}^{(3)}\right),
\end{aligned}
\end{equation}
respectively, where $V_{m}^{(3)}$ is the sequence defined by 
\begin{equation}\label{vq}
VN_{m}^{(3)}=\left\{ 
\begin{array}{ccc}
2-3\textbf{i}+\textbf{j}+2\textbf{k} & \textrm{if} & \mymod{m}{0}{3} \\ 
-3+\textbf{i}+2\textbf{j}-3\textbf{k} & \textrm{if} & \mymod{m}{1}{3} \\ 
1+2\textbf{i}-3\textbf{j}+\textbf{k} & \textrm{if} & \mymod{m}{2}{3}%
\end{array}%
\right. ,
\end{equation}
$\underline{\alpha}=1+2\textbf{i}+4\textbf{j}+8\textbf{k}$ and $\underline{\omega_{1,2}}=1+\omega_{1,2}\textbf{i}+\omega_{1,2}^{2}\textbf{j}+\textbf{k}$.
\end{theorem}
\begin{proof}
Repeated use of (\ref{b1}) in (\ref{eq:8}) enables one to write for $\underline{\alpha}=1+2\textbf{i}+4\textbf{j}+8\textbf{k}$ and $\underline{\omega_{1,2}}=1+\omega_{1,2}\textbf{i}+\omega_{1,2}^{2}\textbf{j}+\textbf{k}$  
\begin{equation}\label{eq:19}
\begin{aligned}
JN_{m}^{(3)}&=J_{m}^{(3)}+J_{m+1}^{(3)}\textbf{i}+J_{m+2}^{(3)}\textbf{j}+J_{m+3}^{(3)}\textbf{k}\\
&=\frac{2}{7}2^{n}-\frac{3+2i\sqrt{3}}{21}\omega_{1}^{n}-\frac{3-2i\sqrt{3}}{21}\omega_{2}^{n}\\
&\ \ +\left(\frac{2}{7}2^{m+1}-\frac{3+2i\sqrt{3}}{21}\omega_{1}^{m+1}-\frac{3-2i\sqrt{3}}{21}\omega_{2}^{m+1}\right)\textbf{i}\\
&\ \ +\left(\frac{2}{7}2^{m+2}-\frac{3+2i\sqrt{3}}{21}\omega_{1}^{m+2}-\frac{3-2i\sqrt{3}}{21}\omega_{2}^{m+2}\right)\textbf{j}\\
&\ \ +\left(\frac{2}{7}2^{m+3}-\frac{3+2i\sqrt{3}}{21}\omega_{1}^{m+3}-\frac{3-2i\sqrt{3}}{21}\omega_{2}^{m+3}\right)\textbf{k}\\
&=\frac{1}{7}\underline{\alpha}2^{m+1}-\frac{3+2i\sqrt{3}}{21}\underline{\omega_{1}}\omega_{1}^{m}-\frac{3-2i\sqrt{3}}{21}\underline{\omega_{2}}\omega_{2}^{m}
\end{aligned} 
\end{equation}
and similarly making use of (\ref{b2}) in (\ref{eq:9}) yields
\begin{equation}\label{eq:20}
\begin{aligned}
jN_{m}^{(3)}&=j_{m}^{(3)}+j_{m+1}^{(3)}\textbf{i}+j_{m+2}^{(3)}\textbf{j}+j_{m+3}^{(3)}\textbf{k}\\
&=\frac{1}{7}\underline{\alpha}2^{m+3}+\frac{3+2i\sqrt{3}}{7}\underline{\omega_{1}}\omega_{1}^{m}+\frac{3-2i\sqrt{3}}{7}\underline{\omega_{2}}\omega_{2}^{m}.
\end{aligned}
\end{equation}
The formulas in (\ref{eq:19}) and (\ref{eq:20}) are called as Binet's formulas for the dual third-order Jacobsthal and dual third-order Jacobsthal-Lucas quaternions, respectively. Using notation in (\ref{vq}), we obtain the results (\ref{t10}) and (\ref{t11}).
\end{proof}

\begin{theorem}[D'Ocagne-like Identity]\label{th:6}
Let $JN_{m}^{(3)}$ be the $m$-th terms of the dual third-order Jacobsthal quaternion sequence $\{JN_{m}^{(3)}\}_{m\geq0}$. In this case, for $n\geq m \geq0$, the d'Ocagne identities for $JN_{m}^{(3)}$ is as follows:
\begin{equation}\label{t12}
JN_{n}^{(3)}JN_{m+1}^{(3)}-JN_{n+1}^{(3)}JN_{m}^{(3)}=\frac{1}{7}\left( 
\begin{array}{ccc} \underline{\alpha}(2^{n+1}UN_{m+1}^{(3)}-2^{m+1}UN_{n+1}^{(3)})\\
+(1-\textbf{i}-\textbf{j}+2\textbf{k})U_{n-m}^{(3)}
\end{array}%
\right),
\end{equation}
\begin{equation}\label{t13}
\left(JN_{m+1}^{(3)}\right)^{2}-JN_{m+2}^{(3)}JN_{m}^{(3)}=\frac{1}{7}\left( 
\begin{array}{ccc} 2^{m+1}\underline{\alpha}(2UN_{m+1}^{(3)}-UN_{m+2}^{(3)})\\
+(1-\textbf{i}-\textbf{j}+2\textbf{k})
\end{array}%
\right),
\end{equation}
where $UN_{m+1}^{(3)}=\frac{1}{7}(2VN_{m}^{(3)}-VN_{m+1}^{(3)})$ and $\underline{\alpha}=1+2\textbf{i}+4\textbf{j}+8\textbf{k}$.
\end{theorem}
\begin{proof}
(\ref{t12}): Using Eqs. (\ref{t10}) and (\ref{vq}), we get
\begin{equation}\label{eq:21}
\begin{aligned}
JN_{n}^{(3)}JN_{m+1}^{(3)}&-JN_{n+1}^{(3)}JN_{m}^{(3)}\\
&=\frac{1}{49}\left(\begin{array}{ccc}(2^{n+1}\underline{\alpha}-VN_{n}^{(3)})(2^{m+2}\underline{\alpha}-VN_{m+1}^{(3)})\\
-(2^{n+2}\underline{\alpha}-VN_{n+1}^{(3)})(2^{m+1}\underline{\alpha}-VN_{m}^{(3)})\end{array}%
\right)\\
&=\frac{1}{49}\left( 
\begin{array}{ccc}
-2^{n+1}\underline{\alpha}VN_{m+1}^{(3)}-2^{m+2}VN_{n}^{(3)}\underline{\alpha}+VN_{n}^{(3)}VN_{m+1}^{(3)}\\
+2^{n+2}\underline{\alpha}VN_{m}^{(3)}+2^{m+1}VN_{n+1}^{(3)}\underline{\alpha}-VN_{n+1}^{(3)}VN_{m}^{(3)}
\end{array}%
\right)\\
&=\frac{1}{7}\left( 
\begin{array}{ccc} (1+2\textbf{i}+4\textbf{j}+8\textbf{k})(2^{n+1}UN_{m+1}^{(3)}-2^{m+1}UN_{n+1}^{(3)})\\
+(1-\textbf{i}-\textbf{j}+2\textbf{k})U_{n-m}^{(3)}
\end{array}%
\right),
\end{aligned}
\end{equation}
where $UN_{m+1}^{(3)}=\frac{1}{7}(2VN_{m}^{(3)}-VN_{m+1}^{(3)})$ and $VN_{m}^{(3)}$ as in (\ref{vq}). In particular, if $n=m+1$ in Eq. (\ref{eq:21}), we obtain for $m\geq 0$,
\begin{equation}\label{eq:22}
\left(JN_{m+1}^{(3)}\right)^{2}-JN_{m+2}^{(3)}JN_{m}^{(3)}=\frac{1}{7}\left( 
\begin{array}{ccc} (1+2\textbf{i}+4\textbf{j}+8\textbf{k})2^{m+1}(2UN_{m+1}^{(3)}-UN_{m+2}^{(3)})\\
+(1-\textbf{i}-\textbf{j}+2\textbf{k})
\end{array}%
\right).
\end{equation}
\end{proof}

We will give an example in which we check in a particular case the Cassini-like identity for dual third-order Jacobsthal quaternions.

\begin{example}\label{eq:23}
Let $JN_{0}^{(3)}$, $JN_{1}^{(3)}$, $JN_{2}^{(3)}$ and $JN_{3}^{(3)}$ be the dual third-order Jacobsthal quaternions such that $JN_{0}^{(3)}=\textbf{i}+\textbf{j}+2\textbf{k}$, $JN_{1}^{(3)}=1+\textbf{i}+2\textbf{j}+5\textbf{k}$, $JN_{2}^{(3)}=1+2\textbf{i}+5\textbf{j}+9\textbf{k}$ and $JN_{3}^{(3)}=2+5\textbf{i}+9\textbf{j}+18\textbf{k}$. In this case,
\begin{align*}
\left(JN_{1}^{(3)}\right)^{2}-JN_{2}^{(3)}JN_{0}^{(3)}&=(1+\textbf{i}+2\textbf{j}+5\textbf{k})^{2}-(1+2\textbf{i}+5\textbf{j}+9\textbf{k})(\textbf{i}+\textbf{j}+2\textbf{k})\\
&=(1+2\textbf{i}+4\textbf{j}+10\textbf{k})-(\textbf{i}+\textbf{j}+2\textbf{k})\\
&=1+\textbf{i}+3\textbf{j}+8\textbf{k}\\
&=\frac{1}{7}\left( 
\begin{array}{ccc} 2(1+2\textbf{i}+4\textbf{j}+8\textbf{k})(2UN_{1}^{(3)}-UN_{2}^{(3)})\\
+(1-\textbf{i}-\textbf{j}+2\textbf{k})
\end{array}%
\right)
\end{align*}
and
\begin{align*}
\left(JN_{2}^{(3)}\right)^{2}-JN_{3}^{(3)}JN_{1}^{(3)}&=(1+2\textbf{i}+5\textbf{j}+9\textbf{k})^{2}-(2+5\textbf{i}+9\textbf{j}+18\textbf{k})(1+\textbf{i}+2\textbf{j}+5\textbf{k})\\
&=(1+4\textbf{i}+10\textbf{j}+18\textbf{k})-(2+7\textbf{i}+13\textbf{j}+28\textbf{k})\\
&=-1-3\textbf{i}-3\textbf{j}-10\textbf{k}\\
&=\frac{1}{7}\left( 
\begin{array}{ccc} 4(1+2\textbf{i}+4\textbf{j}+8\textbf{k})(2UN_{2}^{(3)}-UN_{3}^{(3)})\\
+(1-\textbf{i}-\textbf{j}+2\textbf{k})
\end{array}%
\right)
\end{align*}
\end{example}

\section{Conclusions}\label{sec:3}
There are two differences between the dual third-order Jacobsthal and the dual coefficient third-order Jacobsthal quaternions. The first one is as follows: the dual coefficient third-order Jacobsthal quaternionic units are $\textbf{i}^{2}=\textbf{j}^{2}=\textbf{k}^{2}=\textbf{ijk}=-1$ whereas the dual third-order Jacobsthal quaternionic units are $\textbf{i}^{2}=\textbf{j}^{2}=\textbf{k}^{2}=0,\ \textbf{ij}=-\textbf{ji}=\textbf{jk}=-\textbf{kj}=\textbf{ik}=-\textbf{ki}=0$. The second one is as follows: the elements of the dual coefficient third-order Jacobsthal quaternion are
$J_{m}^{(3)}+\varepsilon J_{m+1}^{(3)}$ ($\varepsilon^{2}=0,\ \varepsilon\neq 0$) whereas the elements of the dual third-order Jacobsthal quaternions are $m$-th third-order Jacobsthal number $J_{m}^{(3)}$.

\medskip
\end{document}